\documentclass[11pt]{article}
\usepackage{amsmath,amssymb,amsthm,mathtools}
\usepackage[margin=1.1in]{geometry}
\usepackage{hyperref}
\usepackage{mathrsfs}

\newtheorem{theorem}{Theorem}
\newtheorem{thmX}{Theorem}

\newtheorem{proposition}{Proposition}
\newtheorem{lemma}[proposition]{Lemma}

\newtheorem{example}[proposition]{Example}
\theoremstyle{definition}
\newtheorem{definition}[proposition]{Definition}
\theoremstyle{remark}
\newtheorem{remark}[proposition]{Remark}

\usepackage{bbm}
\newcommand{\BF}{\bf\boldmath }

\newcommand{\R}{\mathbb{R}}
\newcommand{\bT}{\mathbb{T}}
\newcommand{\bR}{\mathbb{R}}
\newcommand{\bS}{\mathbb{S}}
\newcommand{\HH}{\mathbb{H}}

\newcommand{\Free}{\mathrm{Free}^\infty}
\newcommand{\Imm}{\operatorname{Im}}

\newcommand{\Ree}{\operatorname{Re}}

\newcommand{\spn}{\operatorname{span}}
\def\cD{{\cal D}}

\author{Roberto De Leo}
\title{Free maps in critical dimension on $\bT^m$}
\date{\today}

\begin{document}

\maketitle

\begin{abstract}
The critical dimension for free maps on m-manifolds is $q_m=m(m+3)/2$.
We show that $\Free(\bT^m,\R^{q_m})\neq\emptyset$ for every $m\ge 1$.
The main tool is a product construction: given free maps in critical dimension on 
$\bT^a$ and on $\bT^b$, together with a cross-free map on $\bT^a\times \bT^b$, whose mixed Hessian is invertible everywhere, we obtain a free map in critical dimension on $\bT^{a+b}$, the osculating matrix of the product being block-triangular. 
Cross-free maps are additive in each argument, and in the torus setting none exists when one factor is one-dimensional. 
We construct explicit cross-free maps on $\bT^2\times \bT^2$ and on $\bT^3\times\bT^4$ using quaternionic multiplication. 
Combined with the low-dimensional cases  $m\leq 5$, constructed explicitly by the author in a previous publication, these yield the result by induction.
%The proof combines a block-triangular product decomposition of the osculating matrix with Gromov's h-principle for directed immersions, applied to two open subsets of Grassmannians defined through quaternionic multiplication, whose ampleness we verify. 
%The low-dimensional cases $m\le 5$ explicitly constructed by the author in a previous publication serve as the base of an induction.
\end{abstract}
\maketitle

\section{Introduction}

Throughout, $\bT^m=\R^m/2\pi\mathbb{Z}^m$ with its flat coordinates, and a $C^2$ map
$f\colon \bT^m\to\R^q$ is \emph{free} if the $q_m$ vectors $\partial_\alpha f$ $(1\le \alpha\le m)$ and $\partial_\alpha\partial_\beta f$ $(1\le \alpha\le \beta\le m)$
are linearly independent at every point. For $q=q_m$ this means that the
$q_m\times q_m$ osculating matrix $\cD f$ has nonzero determinant everywhere.

%Within their work on elimination of singularities, 
It was shown by Yakov Eliashberg and Misha Gromov~\cite{EG71} (see also~\cite[Section 1.1.4]{Gro86} and~\cite[Theorem 8.3.3]{CEM24}) that the partial differential relation of being a free map \(M\to \R^q\) satisfies the \(h\)-principle if either \(q>q_m\) or if \(q=q_m\) and \(M\) is open (i.e. has no compact connected component). 
In particular this means that, under those hypotheses, a sufficient condition for free maps \(M\to \R^q\) to arise is that \(M\) be parallelizable.
On the contrary, no general criterion is known about the existence of free maps on closed manifolds in critical dimension besides projective spaces and spheres (e.g. see~\cite[Section~1.2, case 5]{Gro17} and~\cite[Intrigue, case F]{CEM24}).
%(because of the Veronese map, e.g. see~\cite{Gro86}).

In~\cite{DL26} we extended the list of closed manifolds admitting free maps in critical dimension to closed surfaces and low-dimensional tori.
In particular, the main result of this work is based on the following fact, proven by the author by exhibiting concrete maps obtained through an elementary ansatz.

\begin{thmX}[De Leo, 2026~\cite{DL26}]\label{thm:base}
$\Free(\bT^m,\R^{q_m})\neq\emptyset$ for $m=2,3,4,5$.
\end{thmX}

The case $m=1$ is trivial because $\bS^1$ is a sphere and, as recalled above, free maps in critical dimension are known to exist on all spheres. 
The main result of this article is that free maps in critical dimension arise on $\bT^m$ for all integers $m\geq1$ (Theorem~\ref{thm:main}).
We prove this by induction using the free map $F_2\in \Free(\bT^2,\bR^5)$ in~\cite{DL26}
%(which exist by Theorem~\ref{thm:base}) 
as the base of the induction argument.
Also $F_3$ and $F_5$ in~\cite{DL26} are necessary to the proof, while $F_4$ is never used.
%using the existence result above as its starting point.

\section{Free maps on tori}

%In order to use induction, we need a way to 
\paragraph{Combining free maps.}
The first step towards the proof of Theorem~\ref{thm:main} is introducing a method to generate a free smooth map on $\bT^{a+b}$ out of a free smooth map on $\bT^a$ and a free smooth map on $\bT^b$.
The key ingredient we use to this goal is the following type of function.
\begin{definition}
    Given an $a$-dimensional manifold $M$ with coordinates $(x^i)$ and a $b$-dimensional manifold $N$ with coordinates $(y^\alpha)$, we say that 
    $$
%    h:\bT^a\times\bT^b\to\bR^{ab}
    h:M\times N\to\bR^{ab}
    $$
    is {\bf cross-free} if, for every $(x,y)\in M\times N$, its {\bf mixed Hessian}
    \[
%\text{\BF$H_h$}(x,y)\colon \R^a\otimes\R^b\longrightarrow\R^{ab},\quad e_i\otimes e_\alpha\longmapsto \frac{\partial^2 h}{\partial x_i\,\partial y_\alpha}(x,y),
\text{\BF$H_h$}(x,y)\colon T_xM\otimes  T_yN\longrightarrow\R^{ab},\quad
\frac{\partial}{\partial x^i}\otimes\frac{\partial}{\partial y^\alpha} \longmapsto \frac{\partial^2 h}{\partial x^i\,\partial y^\alpha}(x,y),
\]
    is invertible.
    %at every point of $\bT^a\times \bT^b$.
    We denote by
    $$
    \text{\BF$P^\infty(M,N)$}\subset C^\infty(M\times N,\bR^{ab})
    $$ 
    the set of all cross-free functions on $M\times N$.
    %$\bT^a\times\bT^b$.
\end{definition}
Notice that, unlike freeness, cross-freeness is {\em tensorial} and so independent of the coordinate system.
\begin{example}
    Let $h:\mathbb R^a\times \mathbb R^b\to \mathbb R^{ab}$ be the function $h(x,y)=x\otimes y.$
    In coordinates, 
    $$
    h(x,y)=\bigl(x_i y_\alpha\bigr)_{1\le i\le a,\ 1\le \alpha\le b}. 
    $$
    Then $ \frac{\partial^2 h}{\partial x_i\,\partial y_\alpha}$ is exactly the standard basis vector \(E_{i\alpha}\in\mathbb R^{ab}\) and so
    $$ H_h:\mathbb R^a\otimes\mathbb R^b\to\mathbb R^{ab}, \quad e_i\otimes e_\alpha\mapsto E_{i\alpha}, $$
    is the identity after the natural identification $\mathbb R^a\otimes\mathbb R^b\simeq \mathbb R^{ab}$.
    Hence, $h\in P^\infty(\bR^a,\bR^b)$.

    %So \(h(x,y)=x\otimes y\) is the canonical cross-free map.
    For instance, when \(a=b=2\),
$$ h(x_1,x_2,y_1,y_2) = (x_1y_1,x_1y_2,x_2y_1,x_2y_2), $$
and the four mixed derivatives are $
(1,0,0,0),(0,1,0,0),(0,0,1,0),(0,0,0,1).$
\end{example}

\begin{proposition}
    \label{prop: P a,1}
    Assume that $M$ is a closed manifold and let $N$ be either $\bR$ or $\bS^1$.
    Then $P^\infty(M,N)=\emptyset$.
\end{proposition}
\begin{proof}
    Since $b=1$, we have that $P^\infty(M,N)\subset C^\infty(M\times N,\bR^{a})$.
    Then, for any fixed $y\in N$, the Jacobian of the map $M\to\bR^a$ defined by $x\mapsto \partial_y h(x,y)$ is $J=\partial^2_{x^iy} h$, which is invertible at every point when $h\in P^\infty(M,N)$.
    Such a map is a local diffeomorphism, which is impossible when $M$ is closed.
\end{proof}
In this article we are interested in the spaces $P^\infty(\bT^a,\bT^b)$, for which we use throughout the article the notation {\BF $P_{a,b}$}.
Notice that there is a natural identification of $P_{a,b}$ with $P_{b,a}$, so they are empty or non-empty at the same time.

% We start introducing a way to combine two free maps into a free map in higher dimension.

% The first main ingredient of 

% \begin{theorem}\label{thm:main}
% $\Free(\bT^m,\R^{q_m})\neq\emptyset$ for every $m\ge1$.
% \end{theorem}

% %The following cases are known explicitly and are used as input.

% %For $m=1$ this is the circle $x\mapsto(\cos x,\sin x)$; for $m=2$ the verification is by hand; for $m=3,5$ it relies on exact Sturm-sequence computations.
% %We shall not need De Leo's case $m=4$, which also follows from the argument below.

% \section{The product reduction}

% \begin{definition}
% For integers $a,b\ge1$ we say that $P_{a,b}$ holds if there is a smooth map
% $h\colon \bT^a\times \bT^b\to\R^{ab}$, $(x,y)\mapsto h(x,y)$, whose \emph{mixed Hessian}
% \[
% H_h(x,y)\colon \R^a\otimes\R^b\longrightarrow\R^{ab},\quad
% e_i\otimes e_\alpha\longmapsto \frac{\partial^2 h}{\partial x_i\,\partial y_\alpha}(x,y),
% \]
% is invertible at every point of $\bT^a\times \bT^b$.
% \end{definition}

% Clearly $P_{a,b}\Leftrightarrow P_{b,a}$.

\begin{lemma}[Product lemma]\label{lem:product}
    Assume that $a,b$ are such that neither of the sets $$
    P_{a,b},\quad\Free(\bT^a,\R^{q_a}),\quad\Free(\bT^b,\R^{q_b})
    $$ 
    is empty. 
    Let $f\in\Free(\bT^a,\R^{q_a})$, $g\in\Free(\bT^b,\R^{q_b})$ and $h\in P_{a,b}$ and set
\[
F:\bT^{a+b}\to\bR^{q_{a+b}},\quad F(x,y)=\bigl(f(x),\,g(y),\,h(x,y)\bigr).
%\in \R^{q_a}\oplus\R^{q_b}\oplus\R^{ab}=\R^{q_{a+b}}
\]
Then $F\in\Free(\bT^{a+b},\R^{q_{a+b}})$.
\end{lemma}

\begin{proof}
First we verify the dimension count:
$$
q_a + q_b + ab = \tfrac12\bigl[a(a+3) + b(b+3) + 2ab\bigr] = \tfrac12\bigl[(a+b)(a+b+3)\bigr] =q_{a+b}.
$$
% $$
% $q_{a+b}-q_a-q_b=\tfrac12\bigl[(a+b)(a+b+3)-a(a+3)-b(b+3)\bigr]=ab$, so $F$ lands in
% the critical dimension. 
Now, order the columns of $\cD F$ as follows: derivatives of order one and two in $x$
only; derivatives of order one and two in $y$ only; the $ab$ mixed derivatives
$\partial_{x_i}\partial_{y_\alpha}$. Order the rows according to the three target summands.
Since $f$ depends only on $x$, the $f$-rows vanish on the second and third groups of columns;
likewise the $g$-rows vanish on the first and third groups. Hence
\[
\cD F=\begin{pmatrix} \cD f & 0 & 0\\ 0 & \cD g & 0\\ * & * & H_h\end{pmatrix},
\quad \det \cD F=\pm\det \cD f\cdot\det \cD g\cdot\det H_h\neq 0.\qedhere
\]
\end{proof}

\begin{lemma}[Additivity]\label{lem:additive}
    If $P_{a,b_1}\neq\emptyset$ and $P_{a,b_2}\neq\emptyset$, then $P_{a,b_1+b_2}\neq\emptyset$. 
    By symmetry, $P_{a_1,b}\neq\emptyset$ and $P_{a_2,b}\neq\emptyset$ implies $P_{a_1+a_2,b}\neq\emptyset$.
\end{lemma}

\begin{proof}
    Write $y=(y',y'')\in \bT^{b_1}\times \bT^{b_2}$, let $h_1\in P_{a,b_1}$ and $h_2\in P_{a,b_2}$ and set 
    $$
    h(x,y)=(h_1(x,y'),h_2(x,y''))\in\R^{ab_1}\oplus\R^{ab_2}.
    $$ 
    Since $h_1$ does not depend on $y''$ and $h_2$ does not depend on $y'$, after ordering the columns as $(x,y')$-mixed then $(x,y'')$-mixed, $H_h=\operatorname{diag}(H_{h_1},H_{h_2})$.
\end{proof}

\begin{remark}\label{rem:obstruction}
$P_{a,1}=\emptyset$ for every $a$ by Proposition~\ref{prop: P a,1}. 
% Indeed, if $h\colon \bT^a\times \bS^1\to\R^a$ realized it, then for fixed $y$ the map $x\mapsto\partial_y h(x,y)$, $\bT^a\to\R^a$, would have invertible Jacobian $H_h(\cdot,y)$ everywhere, i.e.\ it would be a local diffeomorphism from a compact manifold to $\R^a$, which is impossible (its image would be open and compact). 
% In particular, the case $m=3$ cannot be obtained from the lower-dimensional cases by Lemma~\ref{lem:product}. 
%For $m=5$, this rules out the decomposition $5=1+4$, although a proof of $P_{2,3}$ would yield the case $m=5$ from the cases $m=2$ and $m=3$.
This is why the base case $m=3$ cannot be obtained from Lemma~\ref{lem:product}.
The case $m=5$ could be obtained if $P_{2,3}$ turned out to be non-empty. 
We do not investigate this matter here.
\end{remark}

We prove in the next section by an explicit construction that $P_{2,2}$ and $P_{3,4}$ are not empty.
Below we show that these facts are enough to prove our main theorem.

\begin{theorem}\label{thm:main}
    $\Free(\bT^m,\R^{q_m})\neq\emptyset$ for every $m\ge1$.
\end{theorem}

\begin{proof}
    Notice first that, by Lemma~\ref{lem:additive}, since $P_{2,2}\neq\emptyset$ then $P_{2,2\ell}\neq\emptyset$ for all $\ell\ge1$.
    Since $\Free(\bT^2,\bR^5)\neq\emptyset$, then we get inductively through Lemma~\ref{lem:product} applied to the case $a=2,b=2\ell$ that $\Free(\bT^{m},\bR^{q_{m}})\neq\emptyset$ for every even $m\geq2$.
    
    Now, by symmetry, $P_{4,2}\neq\emptyset$.
    Moreover, since  $P_{3,4}\neq\emptyset$, then by symmetry $P_{4,3}\neq\emptyset$.
    Hence, by Lemma~\ref{lem:additive}, $P_{4,k}\neq\emptyset$ for $k\geq2$.
    Hence, for $m=2\ell+1\geq7$, namely $\ell\geq3$, we get inductively through Lemma~\ref{lem:product} applied to the case $a=4,b=m-4$ that $\Free(\bT^{m},\bR^{q_{m}})\neq\emptyset$ for every odd $m\geq7$. 

    By Theorem~\ref{thm:base}, $\Free(\bT^3,\bR^9)\neq\emptyset$ and $\Free(\bT^5,\bR^{20})\neq\emptyset$, so now all cases $m\geq2$ are covered.
    Finally, the case $m=1$ holds because $\bS^1$ is a sphere.
%     and symmetry, if $P_{2,2}\neq\emptyset$ then $P_{2,2\ell}\neq\emptyset$ for all $\ell\ge1$ and
% $P_{4,2}$; $P_{3,4}$ gives $P_{4,3}$; hence $P_{4,n}$ for every $n\ge2$ (write $n=2\ell$ or
% $n=3+2\ell$). We argue by induction on $m$, the cases $m\le 5$ being Theorem~\ref{thm:base}.
% If $m\ge6$ is even, apply Lemma~\ref{lem:product} with $(a,b)=(2,m-2)$. If $m\ge7$ is odd,
% apply it with $(a,b)=(4,m-4)$, noting $m-4\ge3$.
\end{proof}

% =====================================================================
%  Section 3, rewritten in quaternionic language.
%
%  Matrices survive only where they carry genuine data: the explicit
%  curves r_lambda, G_lambda, S(z), G(z).  Everything else --- L, R, K,
%  L_{e_r}, A, B, Theta, sigma, the immersion --- is written in H.
%
%  Uses the existing macros \R \bT \HH \Imm \Ree \spn.
%  Replaces the whole of Section 3.
% =====================================================================

\section{$P_{2,2}$ and $P_{3,4}$ are not empty}

\subsection*{Quaternionic preliminaries}

Throughout this section we identify $\R^4$ with the quaternions $\HH$ via
\[
  (a_0,a_1,a_2,a_3)\longmapsto a_0+a_1i+a_2j+a_3k ,
\]
and $\R^3$ with the imaginary quaternions $\Imm\HH$. We write
$e_0,e_1,e_2,e_3$ for the standard basis, so that $e_0=1$ and
$(e_1,e_2,e_3)=(i,j,k)$, and we use on $\HH$ the Euclidean product
$\langle a,b\rangle=\Ree(a\bar b)$ together with the orientation of
$(1,i,j,k)$. For imaginary $a,b$ one has
\begin{equation}\label{eq:improd}
  ab=-\langle a,b\rangle+a\times b .
\end{equation}

For $a\in\HH$ let $L_a$ and $R_a$ denote left and right multiplication,
\[
  L_aq=aq,\qquad R_aq=qa .
\]
% and let $\mu:\HH\times\HH\to\HH$ be the multiplication operator, namely
% $$
% \mu(a,b)=ab=L_ab=R_ba.
% $$
The associativity of $\HH$ says precisely that
\begin{equation}\label{eq:assoc}
  L_{ab}=L_aL_b,\qquad R_{ab}=R_bR_a,\qquad L_aR_b=R_bL_a
\end{equation}
for all $a,b\in\HH$: left multiplications compose covariantly, right ones contravariantly, and a left multiplication always commutes with a right one. 
The multiplicativity of the norm gives
\begin{equation}\label{eq:norm}
  |aq|=|qa|=|a|\,|q| ,
\end{equation}
so $L_a$ and $R_a$ are $|a|$ times an orthogonal map and so $\det L_a=\det R_a=\pm|a|^4$.
The function $a\mapsto\det L_a/|a|^4$ is
continuous on the connected set $\HH\setminus\{0\}$, takes values in $\{\pm1\}$ and equals $1$ at $a=1$ since $L_1=\mathrm{id}$. 
Hence,
\begin{equation}\label{eq:det}
  \det L_a=\det R_a=|a|^4
\end{equation}
and, in particular, $L_a,R_a\in SO_4$ when $|a|=1$.

% ; in particular, under the identification of $\HH$ with $\bR^4$, $L_a,R_a$ are represented by $4\times4$ linear matrices.
%Since $\frac{1}{|a|}|L_aq|=|q|$, then $\frac{1}{|a|}L_a\in SO_4$ and similarly for $R_a$.
% Then
% \begin{equation}\label{eq:det}
%   \det L_a=\det R_a=\pm|a|^4 
% \end{equation}
% and, when $|a|=1$, since $L_1=R_1=id$, we have that $L_a,R_a\in SO_4$.
% Indeed, by~\eqref{eq:norm}, these determinants equal $\pm|a|^4$, and the sign
% is constant on the connected set $\HH\setminus\{0\}$, hence equal to its
% value $+1$ at $a=1$.

% We abbreviate
% \[
%   L_{e_r}=L_{e_r},\; r=1,2,3,
%   \quad
%   K=R_i,
% \]
% so that $L_{e_r}K=KL_{e_r}$ by the third identity in~\eqref{eq:assoc}. 
%Note that $K\neq L_{e_1}$: the two agree on $\spn\{1,i\}$ and are opposite on $\spn\{j,k\}$, which is exactly where $\HH$ fails to commute.
Finally, we set
\[
  A_x=L_{e^{xi}},\qquad B_y=R_{e^{yi}}.
\]
These are two commuting circle actions on $\HH$ by elements of $SO_4$ with
\begin{equation}\label{eq:AB}
  A'_x=A_xL_i,\qquad B'_y=B_yR_i,
\end{equation}
and 
%such that the relation
% \[
%   \mu\bigl(A_xa,\,B_yb\bigr)=A_xB_y\,\mu(a,b).
% \]
\[
  (A_xa)(B_yb)=A_x(B_y(ab)).
\]
%is nothing but $(e^{xi}a)(b\,e^{yi})=e^{xi}(ab)e^{yi}$.

\subsection*{The map $\Theta$ and its section}

\begin{definition}
Let
\[
  \Theta:\HH\longrightarrow\Imm\HH,\qquad \Theta(a)=a\,i\,\bar a .
\]
\end{definition}

\begin{lemma}\label{lem:theta}
The map $\Theta$ has the following properties.
\begin{enumerate}
\item $|\Theta(a)|=|a|^2$ and
      $\Theta(B_ya)=\Theta(a)$ for every $y\in\R$; thus,
      $\Theta$ is the Hopf map up to scale.
\item For every smooth curve $v=v(z)$ in $\HH$,
\[
  \tfrac12\tfrac{d}{dz}\Theta(v)=\Imm\bigl(v\,i\,\overline{v'}\bigr),
\]
  equivalently
\[
  \bigl\langle\tfrac12\tfrac{d}{dz}\Theta(v),\,e_r\bigr\rangle
  =\langle L_{e_r}v',R_iv\rangle,\quad r=1,2,3 .
\]
\item {\rm(Section.)} For $\xi\in\Imm\HH$ put $\rho=|\xi|$ and
      $\xi_1=\langle \xi,i\rangle$. 
      On the chart $\rho+\xi_1>0$ set
\[
  \sigma(\xi)=\frac{\rho-\xi i}{\sqrt{2(\rho+\xi_1)}} .
\]
    Then $\sigma$ is smooth and
\[
  \Theta(\sigma(\xi))=\xi,\qquad |\sigma(\xi)|^2=|\xi| .
\]
\end{enumerate}
\end{lemma}

\begin{proof}
First of all we verify that $\Theta(\HH)\subset\Imm\HH$.
This holds because $\overline{\Theta(a)}=\overline{a\,i\,\bar a}=a\,\bar\imath\,\bar a=-\Theta(a)$.

{\bf 1.} $|\Theta(a)|=|a|\,|i|\,|\bar a|=|a|^2$
by~\eqref{eq:norm} and
$\Theta(B_y a)=ae^{y i}\,i\,\overline{ae^{y i}}=a\,e^{y i}ie^{-y i}\,\bar a=a\,i\,\bar a=\Theta(a)$.

\medskip
{\bf 2.} Since $\overline{v\,\bar\imath\,\bar v'}=-v'i\bar v$, we have that
$$
\tfrac{d}{dz}\Theta(v)=v'i\bar v+v\,i\,\overline{v'}=2\Imm(v\,i\,\overline{v'}).
$$

For the second form, using $\langle x,y\rangle=\Ree(x\bar y)$ and
$\overline{vi}=-i\bar v$, we have that
\[
  \langle L_{e_r}v',R_iv\rangle=\Ree\bigl(e_rv'\,\overline{vi}\bigr)
  =-\Ree\bigl(e_rv'i\bar v\bigr)
  =-\langle e_r,\overline{v'\,i\,\overline{v}}\rangle
  =\langle e_r,\Imm(v\,i\,\overline{v'})\rangle=\tfrac{1}{2}\tfrac{d}{dz}\Theta(v).
\]

{\bf 3.} We first point out that, for a unit $u\in\Imm\HH$ and any $v\in\Imm\HH$,
\begin{equation}\label{eq:refl}
  u\,v\,u=v-2\langle u,v\rangle u .
\end{equation}
Indeed, by~\eqref{eq:improd}, 
$$
uvu=\bigl(-\langle u,v\rangle+u\times v\bigr)u
=-\langle u,v\rangle u+(u\times v)u,
$$
and, again by~\eqref{eq:improd},
$$
(u\times v)u=-\langle u\times v,u\rangle+(u\times v)\times u
=v\langle u,u\rangle-u\langle v,u\rangle=v-\langle u,v\rangle u.
$$

Now set $p=\rho-\xi i$, so that $\sigma(\xi)=p/\sqrt{2(\rho+\xi_1)}$. Since
$\Ree(\xi i)=-\langle \xi,i\rangle=-\xi_1$ and $|\xi i|=\rho$,
\[
  |p|^2=\rho^2-2\rho\,\Ree(\xi i)+|\xi i|^2=2\rho(\rho+\xi_1).
\]
Moreover $\bar p=\rho-i\xi$ and $pi=(\rho-\xi i)i=\rho i+\xi$, whence
\[
  \Theta(p)=(\rho i+\xi)(\rho-i\xi)=\rho^2i+2\rho \xi-\xi i\xi .
\]
Writing $\xi=\rho u$ with $u$ unit imaginary and applying~\eqref{eq:refl},
\[
  \xi i\xi=\rho^2\,u\,i\,u=\rho^2\bigl(i-2\langle u,i\rangle u\bigr)=\rho^2i-2\xi_1\xi ,
\]
so that $\Theta(p)=2(\rho+\xi_1)\xi$. Dividing by $|{\cdot}|$ as prescribed,
\[
  \Theta(\sigma(\xi))=\frac{2(\rho+\xi_1)\xi}{2(\rho+\xi_1)}=\xi,
  \qquad
  |\sigma(\xi)|^2=\frac{2\rho(\rho+\xi_1)}{2(\rho+\xi_1)}=\rho .
\]
Finally $\rho+\xi_1=0$ exactly when $\xi$ is a negative multiple of $i$, which
is the single fibre the chart omits.
\end{proof}

\begin{lemma}\label{lem:det}
For all $u_1,u_2,w_1,w_2\in\HH$,
\[
  \det\bigl(u_1w_1,\;u_1w_2,\;u_2w_1,\;u_2w_2\bigr)
  =-\bigl|\Imm(\bar u_1u_2)\times\Imm(w_2\bar w_1)\bigr|^2 .
\]
\end{lemma}

\begin{proof}
    Both sides are polynomial in the entries, so we may assume $u_1\neq0$ and $w_1\neq0$. 
    By~\eqref{eq:det}, applying $L_{u_1^{-1}}R_{w_1^{-1}}$ multiplies the determinant by $|u_1|^{-4}|w_1|^{-4}$, and it carries the four vectors to
\[
  1,\quad w,\quad u,\quad uw,
  \qquad u=u_1^{-1}u_2,\quad w=w_2w_1^{-1} .
\]
    Write $u=u_0+U$ and $w=w_0+W$ with $U,W\in\Imm\HH$. 
    Subtracting suitable multiples of the first column replaces $w$ by $W$ and $u$ by $U$; then, since, by~\eqref{eq:improd}, $uw=u_0w_0-\langle U,W\rangle+u_0W+w_0U+U\times W$, subtracting multiples of these two columns and of the first replaces $uw$ by $U\times W$. 
    As $1$ is orthogonal to $\Imm\HH$,
\[
  \det(1,w,u,uw)=\det{}_3\bigl(W,U,U\times W\bigr)
  =(W\times U)\cdot(U\times W)=-|U\times W|^2 .
\]
    Restoring the factor $|u_1|^4|w_1|^4$ and using $U=\Imm(\bar u_1u_2)/|u_1|^2$ and $W=\Imm(w_2\bar w_1)/|w_1|^2$ gives the claim.
\end{proof}

\subsection*{An explicit element of $P_{2,2}$}

For $\lambda\in(0,1/2)$ let $r_\lambda:\bS^1\to\Imm\HH$ be the loop
\[
  r_\lambda(s)=\bigl(1-2\lambda\sin3s\bigr)
  \bigl( i \cos s + j \sin s + k(\lambda+\sin3s)\bigr) .
\]
%The first two components have constant norm one, which is what will make non-parallelism immediate; t
Thanks to the presence of the positive factor $1-2\lambda\sin3s$, every component of $r_\lambda$ has zero mean over $[0,2\pi]$, so that there exists a loop $\gamma_\lambda:\bS^1\to\Imm\HH$ with $\gamma'_\lambda(s)=r_\lambda(s)$.
The reader can verify that
\[
  \gamma_\lambda(s)=i(\sin s+\tfrac{\lambda}{2}\cos2s+\tfrac{\lambda}{4}\cos4s) + j (
  -\cos s-\tfrac{\lambda}{2}\sin2s+\tfrac{\lambda}{4}\sin4s) + k (  -\tfrac{1-2\lambda^2}{3}\cos3s+\tfrac{\lambda}{6}\sin6s).
\]

Take now $\lambda=\tfrac14$ and $\nu=\tfrac18$, and lift these loops to $\HH$ by
\[
  c(s)=\overline{\sigma\bigl(4i+2\gamma_\lambda(s)\bigr)},
  \qquad
  d(t)=\sigma\bigl(4i+2\gamma_\nu(t)\bigr).
\]
Since $|\langle \gamma_\lambda,i\rangle|\le1+\tfrac{3\lambda}{4}<\tfrac32$, and similarly for $\gamma_\nu$, both arguments of $\sigma$ have $\xi_1\ge1$ and hence $\rho+\xi_1\ge2$; by Lemma~\ref{lem:theta}(3) the curves $c$ and $d$ are smooth and $2\pi$-periodic.
Moreover, $\Theta(\bar c(s))=4i+2\gamma_\lambda(s)$ and $\Theta(d(t))=4i+2\gamma_\nu(t)$, so that
\begin{equation}\label{eq:cd}
  \tfrac12\bigl(\Theta\circ\bar c\bigr)'(s)=r_\lambda(s),
  \qquad
  \tfrac12\bigl(\Theta\circ d\bigr)'(t)=r_\nu(t).
\end{equation}
%\begin{remark}
Passing to $\bar c$ on the left-hand factor is what matches $\Theta\circ\bar c$ to the slot $\Imm(\bar u_1u_2)$ of Lemma~\ref{lem:det}.
%\end{remark}

Now, for $(x,s),(t,y)\in\bT^2$, define
\[
  p(x,s)=A_xc(s)=e^{xi}c(s),\qquad q(t,y)=B_yd(t)=d(t)e^{yi},
\]
and
\[
  h_{2,2}\bigl((x,s),(t,y)\bigr)=p(x,s)\,q(t,y)\in\HH .
\]

\begin{proposition}
    The map $h_{2,2}:\bT^2\times\bT^2\to\HH$ has everywhere invertible mixed Hessian, namely $h_{2,2}\in P_{2,2}$.
\end{proposition}

\begin{proof}
    Since $p$ depends only on $(x,s)$ and $q$ only on $(t,y)$, the four mixed derivatives are, by~\eqref{eq:AB},
\[
  \partial_x\partial_t h_{2,2}=e^{xi}(ic)\,d'\,e^{yi},\qquad
  \partial_x\partial_y h_{2,2}=e^{xi}(ic)(di)\,e^{yi},
\]
\[
  \partial_s\partial_t h_{2,2}=e^{xi}c'\,d'\,e^{yi},\qquad
  \partial_s\partial_y h_{2,2}=e^{xi}c'(di)\,e^{yi}.
\]
    Left multiplication by $e^{xi}$ and right multiplication by $e^{yi}$ lie in $SO_4$, so both factors may be removed. Lemma~\ref{lem:det} applies with
\[
  u_1=ic,\qquad u_2=c',\qquad w_1=d',\qquad w_2=di ,
\]
    and Lemma~\ref{lem:theta}(2) together with~\eqref{eq:cd} evaluates the two imaginary parts,
\[
  \Imm\bigl(\overline{ic}\,c'\bigr)=-\Imm\bigl(\bar c\,i\,c'\bigr)
  =-\tfrac12\bigl(\Theta\circ\bar c\bigr)'=-r_\lambda(s),
  \qquad
  \Imm\bigl(di\,\overline{d'}\bigr)
  =\tfrac12\bigl(\Theta\circ d\bigr)'=r_\nu(t).
\]
    Hence
\[
  \det H_{h_{2,2}}=-\bigl|r_\lambda(s)\times r_\nu(t)\bigr|^2 ,
\]
    so it is enough to show that $r_\lambda(s)$ and $r_\nu(t)$ are never parallel. 
    The positive factors $1-2\lambda\sin3s$ and $1-2\nu\sin3t$ are irrelevant for parallelism, so we get rid of it.
    Suppose now that
\[
  (\cos s,\sin s,\lambda+\sin3s)=\kappa(\cos t,\sin t,\nu+\sin3t).
\]
    The first two coordinates have Euclidean norm one, hence $\kappa=\pm1$. 
    If $\kappa=1$ then $t\equiv s\pmod{2\pi}$ and the third coordinate gives $\lambda=\nu$, a contradiction. 
    If $\kappa=-1$ then $t\equiv s+\pi \pmod{2\pi}$ and the third coordinate gives $\nu=-\lambda$, again a contradiction.
    Therefore, the two vectors are never parallel and so $\det H_{h_{2,2}}<0$ at all points.
\end{proof}

\subsection*{An explicit element of $P_{3,4}$}

We use coordinates $(y_1,y_2,y_3,t)$ on $\bT^4$. 
On $\HH^3\simeq\R^{12}$ we let any $T\in\operatorname{End}(\HH)$ act blockwise,
\[
  \widehat T\,(a_1,a_2,a_3)=(Ta_1,Ta_2,Ta_3),
\]
and we note that $\det\widehat T=(\det T)^3$; in particular $\widehat{L_a}\in SO_{12}$ for $a\in\HH$ with $|a|=1$.

% What the construction requires of the next two displays is only this: a
% loop $z\mapsto\bigl(g_1'(z),g_2'(z),g_3'(z)\bigr)$ of positively oriented
% orthonormal frames of $\Imm\HH$, whose entries all have zero average, so
% that the primitives $g_r$ close up. The following explicit choice does the job. 
Every loop $\zeta:\bS^1\to SO_3$ can be seen as a loop $(\zeta_1(t),\zeta_2(t),\zeta_3(t))$ of positively oriented orthonormal frames of $\Imm\HH$.
When each of the entries of the $\zeta_r$ has zero average, then there is a loop $Z:\bS^1\to M_3(\bR)$ such that $Z'(t)=\zeta(t)$ for every $t$.
The matrix
\[
  G(t)=
  \begin{pmatrix}
    \cos t & -\sin t\cos 2t & \sin t\sin 2t\\
    \sin t & \cos t\cos 2t & -\cos t\sin 2t\\
    0 & \sin 2t & \cos 2t
  \end{pmatrix},
\]
product of a rotation about $k$ through $t$ with a rotation about $i$ through $2t$, satisfies these conditions and is the derivative of
%so that $S(t)\in SO(3)$ and all its entries have teroaverage, and let $g_1,g_2,g_3:\bS^1\to\Imm\HH$ be the rows of
\[
  \Gamma(t)=
  \begin{pmatrix}
    \sin t & \tfrac16\cos3t-\tfrac12\cos t & \tfrac12\sin t-\tfrac16\sin3t\\
    -\cos t & \tfrac12\sin t+\tfrac16\sin3t & \tfrac12\cos t+\tfrac16\cos3t\\
    0 & -\tfrac12\cos2t & \tfrac12\sin2t
  \end{pmatrix}.
\]
Denote by $\Gamma_r$ the rows of $\Gamma$ and set $F_r(t)=3i+2\Gamma_r(t)$ and $v_r(t)=\sigma\bigl(F_r(t)\bigr)$, $r=1,2,3$.
Since $\langle \Gamma_3,i\rangle=0$ and $|\langle \Gamma_1,i\rangle|, |\langle \Gamma_2,i\rangle|\le1$, every $F_r$ satisfies $\langle F_r,i\rangle\ge1$,
so the $v_r$ are smooth periodic curves and, by Lemma~\ref{lem:theta}(3),
\begin{equation}\label{eq:vnorm}
  |v_r(t)|^2=|F_r(t)|\ge1,\qquad |v_3(t)|^2\ge3 .
\end{equation}
Moreover $\Theta(v_r)=F_r=3i+2\Gamma_r$ and $\Gamma'=G$, so Lemma~\ref{lem:theta}(2) gives
\begin{equation}\label{eq:SS}
  \bigl(\langle L_{e_\ell}v_r',R_iv_r\rangle\bigr)_{r,\ell}
  =\Bigl(\bigl\langle\tfrac12\tfrac{d}{dt}\Theta(v_r),e_\ell\bigr\rangle\Bigr)_{r,\ell}
  =\bigl(\langle \Gamma_r'(t),e_\ell\rangle\bigr)_{r,\ell}=G(t).
\end{equation}
Define $q:\bT^4\to\HH^3\simeq\R^{12}$ by
\[
  q(y_1,y_2,y_3,t)
  %)=\bigl(v_1(t)e^{y_1i},v_2(t)e^{y_2i},\,v_3(t)e^{y_3i}\bigr)
  =\bigl(B_{y_r}v_r(t)\bigr)_{r=1,2,3} .
\]

\begin{lemma}\label{lem:twelve}
    Denote $t$ by $y_4$. 
    The twelve vectors $\widehat{L_{e_\ell}}\,q_{y_\alpha}$, $1\le\ell\le3$, $1\le\alpha\le4$ are linearly independent at every point of $\bT^4$. 
    More precisely, up to a fixed sign,
\[
    \det\bigl(\widehat{L_{e_\ell}}\,q_{y_\alpha}\bigr)_{\ell,\alpha}
  =\prod_{r=1}^{3}|v_r(t)|^{2}=\prod_{r=1}^{3}\bigl|F_r(t)\bigr|\;\ge\;3 .
\]
\end{lemma}

\begin{proof}
    By~\eqref{eq:AB}, the derivatives of $q$ are
\[
  q_{y_r}=\bigl(\delta_{r1}\, B_{y_1}R_iv_1,\,\delta_{r2}\, B_{y_2}R_iv_2,\,\delta_{r3}\, B_{y_3}R_iv_3\bigr),\quad r=1,2,3,
  \quad
  q_{y_4}=\bigl(B_{y_r}v_r'\bigr)_{r} .
\]
    Right multiplication by $e^{y_ri}$ in the $r$-th block is an element of $SO_4$ commuting with every $L_{e_\ell}$, and the three of them together form an element of $SO_{12}$; hence they may all be removed, leaving the nine block-supported vectors $L_{e_\ell}R_iv_r$ (one block each) and the three full vectors $\bigl(L_{e_\ell}v_1',L_{e_\ell}v_2',L_{e_\ell}v_3'\bigr)$.

    Now, fix $r$. 
    By~\eqref{eq:norm} the four vectors $e_0w,e_1w,e_2w,e_3w$ are orthogonal of common norm $|w|$ for every $w\in\HH$; with $w=R_iv_r$ this says that $L_{e_1}R_iv_r,L_{e_2}R_iv_r,L_{e_3}R_iv_r$ are orthogonal of common norm $|v_r|$ and span $(R_iv_r)^\perp$. 
    Expanding the determinant along the nine block-supported columns therefore contributes the factor $\prod_{r}|v_r|^{3}$ and leaves the $3\times3$ determinant of the components of $L_{e_\ell}v_r'$ along the unit vectors $R_iv_r/|v_r|$, namely of $\langle L_{e_\ell}v_r',R_iv_r\rangle/|v_r|$. 
    By~\eqref{eq:SS} that determinant equals $\det G(t)\prod_r|v_r|^{-1}=\prod_r|v_r|^{-1}$. 
    Hence, up to a fixed sign,
\[
  \det\bigl(\widehat{L_{e_\ell}}\,q_{y_\alpha}\bigr)_{\ell,\alpha}
  =\prod_{r=1}^{3}|v_r|^{3}\cdot\prod_{r=1}^{3}|v_r|^{-1}
  =\prod_{r=1}^{3}|v_r|^{2},
\]
    and~\eqref{eq:vnorm} bounds this below by $3$.
\end{proof}

We consider the following immersion $\varphi:\bT^3\to\HH$:
\[
  \varphi(x_1,x_2,x_3)=(2+\cos x_3)\,e^{x_1i}+(2+\sin x_3)\,e^{x_2i}j .
\]
Its partial derivatives are
\[
  \varphi_{x_1}=(2+\cos x_3)\,ie^{x_1i},\qquad
  \varphi_{x_2}=(2+\sin x_3)\,e^{x_2i}k,\qquad
  \varphi_{x_3}=-\sin x_3\,e^{x_1i}+\cos x_3\,e^{x_2i}j .
\]
The first lies in $\spn\{1,i\}$ and the second in $\spn\{j,k\}$, so they are
orthogonal and both nonzero; the third is orthogonal to each of them and has
norm one. Thus $\varphi$ is an immersion. Finally define
\[
  h_{3,4}(x,y)=\bigl(\varphi(x)\,q^{(1)}(y),\;\varphi(x)\,q^{(2)}(y),\;
  \varphi(x)\,q^{(3)}(y)\bigr)=\widehat{L_{\varphi(x)}}\,q(y) .
\]

\begin{proposition}
    The map $h_{3,4}:\bT^3\times\bT^4\to\HH^3\simeq\R^{12}$ has everywhere invertible mixed Hessian, namely $h_{3,4}\in P_{3,4}$.
\end{proposition}

\begin{proof}
    Fix $x\in\bT^3$ and let
\[
  \Pi=T_x\varphi\bigl(T_x\bT^3\bigr)
  =\spn\{\varphi_{x_1},\varphi_{x_2},\varphi_{x_3}\}\subset\HH ,
\]
    a $3$-dimensional subspace since $\varphi$ is an immersion. 
    Choose a unit normal $n$ to $\Pi$. Since $L_n$ is orthogonal and $L_n1=n$, it carries $1^\perp=\Imm\HH$ onto $n^\perp=\Pi$, so the three vectors $L_ne_1,L_ne_2,L_ne_3$ form a basis of $\Pi$ and there is an invertible $3\times3$ matrix $C=(C_{h\ell})$ with
\[
  \varphi_{x_h}=\sum_{\ell=1}^{3}C_{h\ell}\,L_ne_\ell=\sum_{\ell=1}^{3}C_{h\ell}\,ne_\ell .
\]
    By associativity,
\[
  \varphi_{x_h}\,q_{y_\alpha}
  =\sum_{\ell=1}^{3}C_{h\ell}\,(ne_\ell)\,q_{y_\alpha}
  =\widehat{L_n}\sum_{\ell=1}^{3}C_{h\ell}\,\widehat{L_{e_\ell}}\,q_{y_\alpha},
\]
    blockwise in each of the three components. 
    Hence the mixed Hessian of $h_{3,4}$ is obtained from the matrix of the twelve vectors $\widehat{L_{e_\ell}}q_{y_\alpha}$ by the invertible changes of basis $\widehat{L_n}$ in the target and $C$ in the $x$-variables, so that
\[
  \det H_{h_{3,4}}
  =\det\widehat{L_n}\cdot(\det C)^4\cdot
  \det\bigl(\widehat{L_{e_\ell}}q_{y_\alpha}\bigr)_{\ell,\alpha}\neq0
\]
 by Lemma~\ref{lem:twelve}.
\end{proof}

\bibliographystyle{unsrt}
\bibliography{refs.bib} 

% \begin{thebibliography}{9}
% \bibitem{DeLeo} R.~De Leo, \emph{Free maps in critical dimension on low-dimensional tori and closed
% surfaces}, preprint, 2026.
% \bibitem{EM} Y.~Eliashberg and N.~Mishachev, \emph{Introduction to the $h$-Principle},
% Graduate Studies in Mathematics 48, AMS, 2002.
% \bibitem{EG} Ya.~M.~Eliashberg and M.~Gromov, Removal of singularities of smooth mappings,
% \emph{Math. USSR Izvestija} 5 (1971), 615--639.
% \bibitem{Gromov} M.~Gromov, \emph{Partial Differential Relations}, Springer, 1986.
% \end{thebibliography}

\end{document}